\documentclass[a4paper]{article}

\usepackage{amsmath}
\usepackage{amssymb}
\usepackage[dvips]{graphicx}
\usepackage[OT4]{fontenc}
\usepackage[latin2]{inputenc}
\usepackage{url}
\usepackage[numbers]{natbib}
\usepackage{bm}
\usepackage{setspace}

\newcommand{\abs}[1]{\left| #1 \right|}

\newcommand{\okra}[1]{\left( #1 \right)}
\newcommand{\kwad}[1]{\left[ #1 \right]}
\newcommand{\klam}[1]{\left\{ #1 \right\}}

\newcommand{\floor}[1]{\left\lfloor #1 \right\rfloor}
\newcommand{\dzi}[1]{\left\langle #1 \right\rangle}

\DeclareMathOperator{\cov}{Cov}
\DeclareMathOperator{\sred}{\mathbf{E}}

\newcommand{\boole}[1]{{\bf 1}_{\klam{#1}}}
\newcommand{\stringminus}{-}
\newcommand{\subseq}{\sqsubseteq}

\newtheorem{definition}{Definition}[section]

\newtheorem{proposition}[definition]{Proposition}
\newtheorem{corollary}[definition]{Corollary}
\newtheorem{lemma}[definition]{Lemma}

\newtheorem{example}[definition]{Example}
\newenvironment*{proof}{\begin{trivlist}\item[]
\noindent\textbf{Proof:}}{$\Box$\par\end{trivlist}}
\newenvironment*{proof*}[1]{\begin{trivlist}\item[]
\noindent\textbf{Proof of #1:}}{$\Box$\par\end{trivlist}}

\author{{\L}ukasz D\k{e}bowski
  \\ 
  {\normalsize
    \url{debowski@cwi.nl} 
  }
  \\
  {\normalsize\textit{ 
      Centrum Wiskunde \& Informatica
  }}
  \\
  {\normalsize\textit{ 
      Science Park 123, NL-1098 XG Amsterdam,
  }}
  \\
  {\normalsize\textit{ 
      The Netherlands
  }}
}

\title{Variable-Length Coding of Two-Sided Asymptotically Mean
  Stationary Measures} \date{}

\begin{document}

\pagestyle{empty}   
\begin{titlepage}
\maketitle

\begin{abstract}
  We collect several observations that concern variable-length coding
  of two-sided infinite sequences in a~probabilistic setting.
  Attention is paid to images and preimages of asymptotically mean
  stationary measures defined on subsets of these sequences. We point
  out sufficient conditions under which the variable length coding and
  its inverse preserve asymptotic mean stationarity. Moreover,
  conditions for preservation of shift-invariant $\sigma$-fields and
  the finite-energy property are discussed, and the block entropies for
  stationary means of coded processes are related in some cases.
  Subsequently, we apply certain of these results to construct
  a~stationary nonergodic process with a~desired linguistic
  interpretation.
  \\[1em]
  \\[1em]
  \\[1em]
  \textbf{Running head}: Variable-Length Coding of Two-Sided AMS Measures
\end{abstract}

\end{titlepage}
\pagestyle{plain}

\section{Introduction}
\label{secIntro}

Let $\mathbb{X}$ and $\mathbb{Y}$ be a~pair of countable sets, called
here
alphabets. 
Fixing $\lambda$ as the empty string, denote the set of nonempty
strings over an alphabet $\mathbb{X}$ as
$\mathbb{X}^+:=\bigcup_{n\in\mathbb{N}} \mathbb{X}^n$ and the set of
all strings as $\mathbb{X}^*:=\mathbb{X}^+\cup\klam{\lambda}$. The set
of one-sided infinite sequences
${x}^{\mathbb{N}}=(x_i)_{i\in\mathbb{N}}=x_1x_2x_3...$ is written
$\mathbb{X}^{\mathbb{N}}$ and the set of two-sided
${x}^{\mathbb{Z}}=(x_i)_{i\in\mathbb{Z}}=...x_{-1}x_{0}\bm{.}x_1x_2...$
is denoted by $\mathbb{X}^{\mathbb{Z}}$. (Mind the
bold-face dot between the $0$-th and the first symbol.) Shorthands
$x^n:=(x_i)_{1\le i\le n}$ and $x_{k}^{l}:=(x_i)_{k\le i\le l}$ denote
substrings, whereas $\abs{x}$ is the length of a~string $x$.

Subsequently, consider a~function
$f:\mathbb{X}\rightarrow\mathbb{Y}^*$ that maps single symbols into
strings. We will extend it to
$f^*:\mathbb{X}^*\rightarrow\mathbb{Y}^*$,
$f^{\mathbb{N}}:\mathbb{X}^{\mathbb{N}}\rightarrow
\mathbb{Y}^{\mathbb{N}}\cup\mathbb{Y}^*$, and
$f^{\mathbb{Z}}:\mathbb{X}^{\mathbb{Z}}\rightarrow
\mathbb{Y}^{\mathbb{Z}}\cup(\mathbb{Y}^*\times\mathbb{Y}^*)$ defined
as
\begin{align}
  \label{Extension}
  f^*(x^n)&:= f(x_1)f(x_2)...f(x_n) ,
  \\
  \label{OneInfExtension}
  f^{\mathbb{N}}({x}^{\mathbb{Z}})&:=
  f(x_1)f(x_2)f(x_3)...
  ,
  \\
  \label{InfExtension}
  f^{\mathbb{Z}}({x}^{\mathbb{Z}})&:=
  ... f(x_{-1})f(x_{0})\bm{.}f(x_1)f(x_2)...
  ,
\end{align}
where $x_i\in\mathbb{X}$. These extensions are known in literature
under several names, such as ``variable-length coding'' \cite{Salomon07}
or ``sequence morphisms'' \cite{AlloucheShallit03}.
The finite extension (\ref{Extension}) plays a~fundamental role in the
definition of instantaneous codes in information theory
\cite{CoverThomas91}. On the other hand, probabilistic analyses that
involve strong laws and ergodic theorems necessarily operate on
infinite sequences, cf., e.g.,
\cite{Shields96,Kallenberg97,Krengel85}. For these analyses,
extensions (\ref{OneInfExtension}) and (\ref{InfExtension}) seem more
natural, and the variable-length coding (\ref{InfExtension}) has been
discussed by communication engineers for a~few decades
\cite{CariolaroPierobon77,GrayKieffer80,TimoBlackmoreHanlen2007}.
%

Fix a~sufficiently rich probability space $(\Omega,\mathcal{J},P)$,
and let $(\mathbb{X}^{\mathbb{Z}},\mathcal{X}^{\mathbb{Z}})$ and
$(\mathbb{Y}^{\mathbb{Z}},\mathcal{Y}^{\mathbb{Z}})$ denote the
standard measurable spaces of two-sided infinite sequences.  Let us
consider a~``shrunk'' stochastic process
$(X_i)_{i\in\mathbb{Z}}:(\Omega,\mathcal{J})\rightarrow
(\mathbb{X}^{\mathbb{Z}},\mathcal{X}^{\mathbb{Z}})$ and
an~``expanded'' process
$(Y_i)_{i\in\mathbb{Z}}:(\Omega,\mathcal{J})\rightarrow
(\mathbb{Y}^{\mathbb{Z}},\mathcal{Y}^{\mathbb{Z}})$ related through an
almost sure equality
\begin{align}
  \label{DefY}
  (Y_i)_{i\in\mathbb{Z}}=f^{\mathbb{Z}}((X_i)_{i\in\mathbb{Z}})
  ,
\end{align}
assuming that $\lim_n \abs{f^*(X_{-m}^n)}=\lim_m
\abs{f^*(X_{-m}^n)}=\infty$ almost surely.  Throughout the article,
the distributions of these processes will be written as
\begin{align}
  \mu=P((X_i)_{i\in\mathbb{Z}}\in\cdot)
  \quad
  \text{ and }
  \quad
  \nu=P((Y_i)_{i\in\mathbb{Z}}\in\cdot)=\mu\circ \okra{f^{\mathbb{Z}}}^{-1}
  .
\end{align}

Having denoted the shift operation as
$T({x}^{\mathbb{Z}}):=...x_{0}x_{1}\bm{.}x_2x_3...=(x_{i+1})_{i\in\mathbb{Z}}$,
a~measure $\mu$ on
$(\mathbb{X}^{\mathbb{Z}},\mathcal{X}^{\mathbb{Z}})$ is called
asymptotically mean stationary (AMS) if the limits
\begin{align}  
  \label{BarMu}
  \bar\mu(A)=\lim_{n\rightarrow\infty} \frac{1}{n}  \sum_{i=0}^{n-1}
  \mu\circ T^{-i}(A)
\end{align}
exist for all $A\in\mathcal{X}^{\mathbb{Z}}$, see
\cite{GrayKieffer80,Krengel85}. The limit $\bar\mu$, if it exists as
as a~total function $\mathcal{X}^{\mathbb{Z}}\rightarrow \mathbb{R}$,
forms a~stationary measure on
$(\mathbb{X}^{\mathbb{Z}},\mathcal{X}^{\mathbb{Z}})$, i.e.,
$\bar\mu\circ T^{-1}=\bar\mu$, and is called the stationary mean of
$\mu$.  It is possible that limits (\ref{BarMu}) exist for
a~computable measure $\mu$ and all cylinder sets but they do not exist
for some other sets, see Example \ref{exChampernowne} later.

As we shall show, under mild conditions, the transported measure
$\nu=\mu\circ \okra{f^{\mathbb{Z}}}^{-1}$ is AMS if $\mu$ is AMS.
A~weaker proposition, assuming a~stationary $\mu$, was obtained in
\cite[Example 6]{GrayKieffer80}.  Besides coding theory, stationary
means $\overline{\mu\circ \okra{f^{\mathbb{Z}}}^{-1}}$ of
variable-length coded measures appear in disguise in statistical
applications such as length-biased sampling \cite{Cox62} or
philosophical probabilistic puzzles such as the Sleeping Beauty
problem \cite{Elga00}.

An application at the interface between information theory and
linguistics has drawn our attention to the question whether a~few
specific properties of a~stationary measure $\mu$ can be
simultaneously preserved by the stationary mean $\overline{\mu\circ
  \okra{f^{\mathbb{Z}}}^{-1}}$ for a~certain injection
$f^{\mathbb{Z}}:\mathbb{X}\rightarrow\mathbb{Y}^*$, where $\mathbb{X}$
is infinite, and $\mathbb{Y}$ is finite. In this article, we gather
several results of independent interest that concern partly relaxed
and partly more general cases of our original problem.  The question
that stimulated our research will be presented at the end of this
section
and answered in positive later.

We shall not discuss measures on one-sided sequences, see
\cite{Rechard56,GrayKieffer80}, since they do not arise naturally in
the methods and applications considered here. However, there appear
a~few more specific conditions on the coding function
$f^{\mathbb{Z}}$, which appeal to two-sidedness of coded sequences.
The first condition has to do with various concepts of
synchronization, see
\cite{Stiffler71,CapocelliOther92,AhlswedeOther05}.
\begin{definition}
  A~function
  $\pi:\mathbb{X}^{\mathbb{Z}}\rightarrow\mathbb{Y}^{\mathbb{Z}}$ is
  called a~\emph{synchronizable injection} if $\pi$ is an injection
  and $T^i\pi({x}^{\mathbb{Z}})=\pi({b}^{\mathbb{Z}})$ for an
  $i\in\mathbb{Z}$ implies $T^j{x}^{\mathbb{Z}}={b}^{\mathbb{Z}}$ for
  some $j\in\mathbb{Z}$.
\end{definition} 
For example, $f^{\mathbb{Z}}$ is a~synchronizable injection for
a~comma-separated code $f(x)=g(x)c$, where $c\in\mathbb{Y}^m$, and
$g:\mathbb{X}\rightarrow (\mathbb{Y}^m\setminus\klam{c})^*$ is an
injection.

Other conditions considered are more local.  Let us recall that
a~set of strings $\mathcal{L}\subset \mathbb{Y}^*$ is called
(i) \emph{prefix-free} if $w\neq zs$ for $w,z\in \mathcal{L}$ and
$s\in \mathbb{Y}^+$,
(ii) \emph{suffix-free} if $w\neq sz$ for $w,z\in \mathcal{L}$ and
$s\in \mathbb{Y}^+$,
(iii) \emph{fix-free} if it is both prefix-free and suffix-free, and
(iv) \emph{complete} if it satisfies the Kraft equality $\sum_{w\in
  \mathcal{L}} \abs{\mathbb{Y}}^{-\abs{w}}=1$, where
$\abs{\mathbb{Y}}$ is the cardinality of $\mathbb{Y}$.
\begin{definition}
  A~function $f:\mathbb{X}\rightarrow\mathbb{Y}^*$ is called
  \emph{(complete) prefix/suffix/fix-free} if $f$ is an injection and
  the image $f(\mathbb{X})$ is respectively (complete)
  prefix/suffix/fix-free. For finite $f(\mathbb{X})$, $f$ is called
  finite.
\end{definition}
For instance, the set $\klam{ 01, 000, 100, 110, 111, 0010, 0011,
  1010, 1011 }$ is complete fix-free with respect to
$\mathbb{Y}=\klam{0,1}$
\cite{GillmanRivest95,AhlswedeBalkenholKhachatrian96}.  The
aforementioned comma-separated code $f(x)=g(x)c$ is prefix-free but it
is not complete.

The main results of this paper are as follows:
\begin{enumerate}
\item The measure $\mu\circ \okra{f^{\mathbb{Z}}}^{-1}$ is AMS for an AMS
  measure $\mu$, provided that the expansion rate $\bar
  l({x}^{\mathbb{Z}}):=\lim_{n} n^{-1}\sum_{i=1}^{n}\abs{f(x_i)}$ is
  in the range $(0,\infty)$ $\mu$-almost everywhere (Section
  \ref{secExpanded}). This result generalizes \cite[Example
  6]{GrayKieffer80}, where $\mu\circ \okra{f^{\mathbb{Z}}}^{-1}$ was
  shown AMS provided that $\mu$ is stationary, $f^{\mathbb{Z}}$ is an
  injection, and the $\mu$-expectation of $\bar l$ is finite.
\item The shift-invariant algebras for processes
  $(X_i)_{i\in\mathbb{Z}}$ and
  ${f^{\mathbb{Z}}}((X_i)_{i\in\mathbb{Z}})$ remain in one-to-one
  correspondence, and their distributions coincide on these algebras
  if ${f^{\mathbb{Z}}}$ is a~synchronizable injection (Section
  \ref{secInvariant}).
\item The measure $\nu\circ {f^{\mathbb{Z}}}$ is stationary or AMS 
  respectively for a~stationary or an AMS measure $\nu$ 
  if $f$ is complete fix-free (Section \ref{secShrunk}).
\item Write the cylinder sets as
  $[u]:=\klam{{x}^{\mathbb{Z}}: x^{\abs{u}}=u}$.  As defined in
  \cite{Shields97}, a~measure $\mu$ \emph{has finite energy} if
  conditional probabilities of cylinder sets are uniformly
  exponentially damped, i.e., if
\begin{align}
  \label{FE}
  \mu([uv])\le Kc^{\abs{v}} \mu([u])
\end{align}
for certain $c<1$ and $K<\infty$.  (Condition (\ref{FE}) may be only
satisfied for $c\ge \abs{\mathbb{X}}^{-1}$ and, for a~finite alphabet
$\mathbb{X}$, (\ref{FE}) implies that the length of the longest
nonoverlapping repeat in the $\mu$-distributed block of length $n$ is
almost surely bounded by $O(\log n)$ \cite{Shields97}.) We will show
that the stationary mean $\bar\mu$ has also finite energy if
(\ref{FE}) holds. Moreover, $\mu\circ \okra{f^{\mathbb{Z}}}^{-1}$
has finite energy if $\mu$ has finite energy and $f$ is finite
prefix-free (Section \ref{secEnergy}).
\item Block entropy for a~measure $\mu$ on
  $(\mathbb{X}^{\mathbb{Z}},\mathcal{X}^{\mathbb{Z}})$ is the function
  \begin{align}
    \label{BlockEntropy}
    H_\mu(i;n):= 
    - \sum_{u\in\mathbb{X}^n} 
    \mu(T^{-i}[u]) \log \mu(T^{-i}[u])
    ,
  \end{align}
  where we also use the shorthand $H_\mu(n):=H_\mu(0;n)$.  We will
  demonstrate that for a~fixed length injection
  $f:\mathbb{X}\rightarrow \mathbb{Y}^K$, a~finite $\mathbb{X}$, and
  $\nu=\mu\circ \okra{f^{\mathbb{Z}}}^{-1}$, block entropies
  $H_{\bar\mu}(n)$ and $H_{\bar\nu}(nK)$ of the stationary means do
  not differ more than a~constant (Section \ref{secEntropy}).
\end{enumerate} 

We have researched these topics while seeking for a~class of
nonergodic processes $(\bar Y_i)_{i\in\mathbb{Z}}$ that satisfy four
conditions:
\begin{itemize}
\item[(a)] $(\bar Y_i)_{i\in\mathbb{Z}}$ is a~process over a~finite
  alphabet $\mathbb{Y}=\klam{0,1,...,D-1}$,
\item[(b)] $(\bar Y_i)_{i\in\mathbb{Z}}$ is stationary,
\item[(c)] $(\bar Y_i)_{i\in\mathbb{Z}}$ has finite energy, and
\item[(d)] there exist independent equidistributed binary random
  variables $(\bar Z_k)_{k\in\mathbb{N}}$, $\bar P(\bar Z_k=z)=1/2$,
  $z\in\klam{0,1}$, measurable against the shift-invariant
  $\sigma$-field of $(\bar Y_i)_{i\in\mathbb{Z}}$ such that
  \begin{align}
    \label{UDPPowerLawY}
    \liminf_{n\rightarrow\infty} n^{-\beta}\abs{\bar U_{\bar\delta}(n)}>0
  \end{align}
  holds for a~certain $\beta\in (0,1)$, all ${\bar\delta}\in(1/2,1)$,
  and the sets $\bar U_{\bar\delta}(n):= \klam{k\in\mathbb{N}: \bar
    P\okra{\bar s_{k}\okra{\bar Y^n}=\bar Z_k}\ge {\bar\delta}}$ of
  well-predictable $\bar Z_k$'s, where functions $\bar s_{k}$ satisfy
  \begin{align}
  \label{UDPcondY}
  \lim_{n\rightarrow\infty} \bar P\okra{\bar s_{k}\okra{\bar
      Y_{i+1}^{i+n}}=\bar Z_k}=1
  ,
  \quad
  \forall i\in\mathbb{Z}
  .
\end{align}
\end{itemize}

As demonstrated in 
\cite{Debowski08c}, 
properties (a)--(d) imply a~power-law growth of the number of distinct
nonterminal symbols in the shortest grammar-based compression of the
block $\bar Y_{i+1}^{i+n}$, see \cite{KiefferYang00,CharikarOthers05}.
In a~linguistic interpretation posited in manuscript
\cite{Debowski08c}, variables $\bar Y_i$ stand for consecutive letters
of an infinitely long text, whereas the values of variables $\bar Z_i$
stand for random facts repetitively described in the text. Since
nonterminal symbols of grammar-based compressed texts in natural
language often correspond to words of a~particular language (we mean
words in the common sense of strings of letters separated by spaces),
the demonstrated implication forms a~new explanation of a power-law
growth of text vocabulary (known as Zipf's law in linguistics
\cite{Zipf65}). Precisely, the explanation takes the form of the
statement: If an $n$-letter long text describes $n^\beta$ independent
facts in a~repetitive way, then the text contains at least
$n^\beta/\log n$ different words.

Properties (b)--(d), but not (a), are satisfied by the following
process.
\begin{example}[\cite{Debowski09}]
  \label{exUDP}
  Let $(X_i)_{i\in\mathbb{Z}}$ be a~process on
  $(\Omega,\mathcal{J},P)$, where variables
\begin{align}
  \label{UDPX}
  X_i=(K_i,Z_{K_i})
\end{align}
assume values from an infinite alphabet
$\mathbb{X}=\mathbb{N}\times\klam{0,1}$, variables $K_i$ and $Z_k$ are
probabilistically independent, $K_i$ are distributed according to
a~power law, $P(K_i=k)=k^{-\alpha}/\zeta(\alpha)$,
$\alpha>1$, $\zeta(\alpha):=\sum_{k=1}^\infty k^{-\alpha}$,
and $Z_k$ are equidistributed, $P(Z_k=z)=1/2$, $z\in\klam{0,1}$. 
\end{example}

Let us write $u\subseq v$ when a~sequence or a~string $v$ contains
a~string $u$ as a~substring.  For
$\mathbb{X}=\mathbb{N}\times\klam{0,1}$ and
$v\in\mathbb{X}^{\mathbb{Z}}\cup\mathbb{X}^*$, define the predictors
\begin{align*}
  s_{k}(v):=
  \begin{cases}
    0 & \text{if $(k,0)\subseq v$ and $(k,1)\not\subseq v$}, \\
    1 & \text{if $(k,1)\subseq v$ and $(k,0)\not\subseq v$}, \\
    2 & \text{else}.
  \end{cases}
\end{align*}
Variables $Z_k$ are measurable against
the shift-invariant $\sigma$-field of $(X_i)_{i\in\mathbb{Z}}$ since
they satisfy $Z_k=s_k((X_i)_{i\in\mathbb{Z}})$ almost
surely. Moreover, 
\begin{align}
  \label{UDPcondX}
  \lim_{n\rightarrow\infty} P\okra{s_{k}\okra{X_{i+1}^{i+n}}=Z_k}=1
  ,
  \quad
  i\in\mathbb{Z}
  ,
\end{align}
and
\begin{align}
  \label{UDPPowerLawX}
  \abs{U_\delta(n)}\ge \kwad{\frac{n}{-\zeta(\alpha)\log(1-\delta)}}^{1/\alpha}
\end{align}
for $\delta\in(1/2,1)$ and $U_\delta(n):= \klam{k\in\mathbb{N}:
  P\okra{s_{k}\okra{X^{n}}=Z_k}\ge \delta}$, as shown in
\cite{Debowski08c}.

We have supposed that a~suitable distribution over a~finite alphabet
can be constructed as the stationary mean of a~certain encoding of the
process (\ref{UDPX}). The results of Sections \ref{secExpanded}
through \ref{secEntropy} suggest the following statement:
\begin{proposition}
  \label{theoConj}
  Let $\mu=P((X_i)_{i\in\mathbb{Z}}\in\cdot)$ be the distribution of
  the process from Example \ref{exUDP} and put
  $\mathbb{Y}=\klam{0,1,2}$. Consider the coding function
  $f:\mathbb{X}\mapsto \mathbb{Y}^+$ given as
  \begin{align}
    \label{ConjCode}
    f(k,z)=b(k)z2
    ,
  \end{align}
  where $1b(k)\in\klam{0,1}^+$ is the binary representation of
  a~natural number $k$. The process $(\bar Y_i)_{i\in\mathbb{Z}}$
  distributed according to the stationary mean $\bar P((\bar
  Y_i)_{i\in\mathbb{Z}}\in\cdot)=\overline{\mu\circ
    \okra{f^{\mathbb{Z}}}^{-1}}$ satisfies conditions (a)--(d) for
  $\beta=\alpha^{-1}$ and $\zeta(\alpha)>4$. Variables
  $\bar Z_k$ may be constructed as $\bar Z_k=\bar s_k((\bar
  Y_i)_{i\in\mathbb{Z}})$, where
  \begin{align}
    \bar s_k(w)
    :=
    \begin{cases}
      0 & \text{if $2b(k)02\subseq w$ 
        and $2b(k)12\not\subseq w$}, \\
      1 & \text{if $2b(k)12\subseq w$ 
        and $2b(k)02\not\subseq w$}, \\
      2 & \text{else}
    \end{cases} 
  \end{align}
  for $w\in\mathbb{Y}^{\mathbb{Z}}\cup\mathbb{Y}^*$.
\end{proposition} 
This proposition is proved in the final Section \ref{secLinguistic}.
The inequality $\zeta(\beta^{-1})>4$ holds for $\beta>0.7728...$.
Mind that processes $(\bar Y_i)_{i\in\mathbb{Z}}$ and
$(X_i)_{i\in\mathbb{Z}}$ live on different probability spaces, say
$(\Omega,\mathcal{J},\bar P)$ and $(\Omega,\mathcal{J},P)$,
respectively.  Obviously, the coding function (\ref{ConjCode}) is
prefix-free, and its extension $f^{\mathbb{Z}}$ is a~synchronizable
injection.

\section{AMS measures and finite expansion rate}
\label{secExpanded}

Previous accounts of AMS measures on two-sided sequences can be found
in \cite{GrayKieffer80,FontanaGrayKieffer81}. Let us recall a~few
useful facts. First of all, for the shift
$T({x}^{\mathbb{Z}}):=...x_{0}x_{1}\bm{.}x_2x_3...=(x_{i+1})_{i\in\mathbb{Z}}$,
an AMS measure $\mu$ on
$(\mathbb{X}^{\mathbb{Z}},\mathcal{X}^{\mathbb{Z}})$ can be
equivalently characterized as such that the almost sure ergodic
theorem is satisfied, i.e., the limit $\lim_{n} n^{-1}
\textstyle\sum_{i=0}^{n-1} g\circ T^i$ exists $\mu$-almost everywhere
for every nonnegative measurable function $g:
(\mathbb{X}^{\mathbb{Z}},\mathcal{X}^{\mathbb{Z}})\rightarrow
(\mathbb{R},\mathcal{R})$ \cite[Theorem 1]{GrayKieffer80}.  Trivially,
$\bar\mu=\mu$ for a~stationary $\mu$. However, the equality
\begin{align}
  \label{EqInvar}
  \bar\mu(A)=\mu(A)
  ,
  \quad
  A\in \mathcal{I}_{\mathbb{X}}
  ,
\end{align}
is also satisfied for the $T$-invariant algebra
$\mathcal{I}_{\mathbb{X}}:=\klam{A\in \mathcal{X}^{\mathbb{Z}}:
  T^{-1}A=A}$ in the general AMS case.  This follows directly from
(\ref{BarMu}), see \cite{GrayKieffer80}. Extending the concept of
ergodicity, usually discussed for stationary measures, an AMS measure
$\mu$ is called \emph{ergodic} if $\mu(A)\in\klam{0,1}$ for all
$A\in\mathcal{I}_{\mathbb{X}}$.

The lemma below is mostly a~well known fact:
\begin{lemma}[\mbox{cf.\ \cite[Theorem 0]{FontanaGrayKieffer81}}]
  \label{theoDomin}
  A~measure $\mu$ on
  $(\mathbb{X}^{\mathbb{Z}},\mathcal{X}^{\mathbb{Z}})$ is AMS if and
  only if there exists a~stationary measure $\tau$ on
  $(\mathbb{X}^{\mathbb{Z}},\mathcal{X}^{\mathbb{Z}})$ such that
  $\tau\gg \mu$. In the latter case, we have $\tau\gg \bar\mu\gg \mu$.
  \\
  Remark: \emph{The notation $\tau\gg \mu$ stands for measure
    dominance, i.e., $\tau(A)=0$ implies that $\mu(A)=0$ for all sets
    $A$ in the domain of $\mu$.}
\end{lemma}
The proof in \cite{FontanaGrayKieffer81} does not cover the inequality
$\tau\gg \bar\mu$. To justify it, let us observe that $\tau\gg \mu$
and $\tau(A)=0$ imply $\mu(T^{-i}A)=0$. Hence $\bar\mu(A)=0$ as well.
Moreover, the proof in \cite{FontanaGrayKieffer81} cannot be carried
to the one-sided case since it applies invariant sets of form
$\bigcup_{i\in\mathbb{Z}} T^iA$. The same trick resurfaces in
Proposition \ref{theoBarNu} below and in Section \ref{secInvariant},
where synchronizable injections are considered.

By the definition, $\bar\mu(A) =\lim_n \int \kwad{n^{-1}
  \sum_{i=0}^{n-1} \boole{T^{i}{x}^{\mathbb{Z}}\in A} }
d\mu({x}^{\mathbb{Z}})$. Hence a~useful frequency interpretation
follows by the dominated convergence.
\begin{lemma}
  \label{theoBarMu}
  For an AMS measure $\mu$ on
  $(\mathbb{X}^{\mathbb{Z}},\mathcal{X}^{\mathbb{Z}})$,
  \begin{align}
    \label{FreqBarMu}
    \bar\mu(A)
    =\int \kwad{\lim_{n\rightarrow\infty} \frac{1}{n}  \sum_{i=0}^{n-1}
      \boole{T^{i}{x}^{\mathbb{Z}}\in A} } d\mu({x}^{\mathbb{Z}})
    ,
    \quad
    A\in \mathcal{X}^{\mathbb{Z}}
    .
  \end{align}
  Remark: \emph{In the ergodic case, the integrated expression is
    almost everywhere constant. Symbol $\boole{\varphi}$ therein
    denotes the indicator function, i.e., $\boole{\varphi}:=1$ if
    $\varphi$ is true and $\boole{\varphi}:=0$ otherwise.}
\end{lemma}

Now we move on to variable-length coding of an AMS measure $\mu$ on
$(\mathbb{X}^{\mathbb{Z}},\mathcal{X}^{\mathbb{Z}})$. Let
$f:\mathbb{X}\rightarrow\mathbb{Y}^*$,
$l_i({x}^{\mathbb{Z}}):=\abs{f(x_i)}$, and
$S({x}^{\mathbb{Z}},n):=\sum_{i=1}^{n} l_i({x}^{\mathbb{Z}})$. By the
ergodic theorem \cite[Theorem 9.6]{Kallenberg97} and Lemma
\ref{theoDomin}, the limit
\begin{align}
  \label{ExpRate}
  \bar l({x}^{\mathbb{Z}}):=\lim_{n\rightarrow\infty} \frac{S({x}^{\mathbb{Z}},n)}{n}
  =\sred_{\bar\mu}\okra{l_1\mid\mid
    \mathcal{I}_{\mathbb{X}}}({x}^{\mathbb{Z}})
\end{align}
exists both $\bar\mu$- and $\mu$-almost everywhere. We will call the
function $\bar l(\cdot)$ the expansion rate, whereas its expectation
will be denoted by
\begin{align}
  \label{ExpExpRate}
  L:=\int \bar l d\mu=\int \bar l d\bar\mu=\int l_1
  d\bar\mu .
\end{align}

Let $T:(y_i)_{i\in\mathbb{Z}}\mapsto (y_{i+1})_{i\in\mathbb{Z}}$ also
denote the shift $\mathbb{Y}^{\mathbb{Z}} \rightarrow
\mathbb{Y}^{\mathbb{Z}}$. Put
\begin{align*}
  F(A,k,{x}^{\mathbb{Z}})&:=\boole{T^{k} f^{\mathbb{Z}}({x}^{\mathbb{Z}})\in A}
  ,
  \\
  G(A,{x}^{\mathbb{Z}})&:=\sum_{k=0}^{\abs{f(x_1)}-1} F(A,k,{x}^{\mathbb{Z}})
  ,
\end{align*}
where $\sum_{k=0}^{-1}:=0$.
By the quasiperiodic identity
\begin{align*}
  T^{\abs{f(x_1)}}f^{\mathbb{Z}}({x}^{\mathbb{Z}})=f^{\mathbb{Z}}(T{x}^{\mathbb{Z}})
  ,
\end{align*}
we have
\begin{align}
  \label{FAshifted}
  F(T^{-1}A,k,{x}^{\mathbb{Z}})&=F(A,k+1,{x}^{\mathbb{Z}})
  ,
  \\
  \label{FXshifted}
   F(A,\abs{f(x_1)},{x}^{\mathbb{Z}})&=F(A,0,T{x}^{\mathbb{Z}})
  ,
  \\
  \label{GAshifted}
  G(T^{-1}A,{x}^{\mathbb{Z}})&=
  G(A,{x}^{\mathbb{Z}})-F(A,0,{x}^{\mathbb{Z}})+F(A,0,T{x}^{\mathbb{Z}})
  .
\end{align}

The following proposition is a~direct consequence of the above
identities and the two previous lemmas.
\begin{proposition}
  \label{theoBarNu}
  Let $\mu$ be an AMS measure on
  $(\mathbb{X}^{\mathbb{Z}},\mathcal{X}^{\mathbb{Z}})$ and suppose
  that the expansion rate (\ref{ExpRate}) is $\mu$-almost surely in
  the range $(0,\infty)$ for an $f:\mathbb{X}\rightarrow\mathbb{Y}^*$.
  Then $\nu=\mu\circ \okra{f^{\mathbb{Z}}}^{-1}$ is an AMS measure on
  $(\mathbb{Y}^{\mathbb{Z}},\mathcal{Y}^{\mathbb{Z}})$ with the
  stationary mean
  \begin{align}
    \label{FreqBarNu}
    \bar\nu(A) 
    &=\int \kwad{\lim_{n\rightarrow\infty}
      \frac{1}{S({x}^{\mathbb{Z}},n)} \sum_{k=0}^{S({x}^{\mathbb{Z}},n)-1}
      F(A,k,{x}^{\mathbb{Z}}) } 
    d\mu({x}^{\mathbb{Z}})
    \\
    &=\int \kwad{\bar l({x}^{\mathbb{Z}})}^{-1} \kwad{\lim_{n\rightarrow\infty}
      \frac{1}{n}\sum_{i=0}^{n-1} G(A,T^i{x}^{\mathbb{Z}}) } 
    d\mu({x}^{\mathbb{Z}})
    \\
    \label{FreqBarNuLast}
    &=\int \kwad{\bar l({x}^{\mathbb{Z}})}^{-1} \kwad{\lim_{n\rightarrow\infty}
      \frac{1}{n}\sum_{i=0}^{n-1} G(A,T^i{x}^{\mathbb{Z}}) } 
    d\bar\mu({x}^{\mathbb{Z}})
    ,
    \quad
    A\in \mathcal{Y}^{\mathbb{Z}}
    .
  \end{align}
\end{proposition}
\begin{proof}
  The transported measure $\nu$ is a~measure on
  $(\mathbb{Y}^{\mathbb{Z}},\mathcal{Y}^{\mathbb{Z}})$ if and only if
  $\lim_n S({x}^{\mathbb{Z}},n)=\infty$ $\mu$-almost surely. This
  condition is satisfied.  Observe that the limits in the brackets in
  (\ref{FreqBarNu})--(\ref{FreqBarNuLast}) exist $\mu$- and
  $\bar\mu$-almost surely. Consequently, the integrals are equal.
  Denote the right-hand side of (\ref{FreqBarNu}) as $\tau(A)$.  The
  function $\tau$ is a~stationary measure on
  $(\mathbb{Y}^{\mathbb{Z}},\mathcal{Y}^{\mathbb{Z}})$ by the
  dominated convergence and the Vitali-Hahn-Saks theorem. Suppose that
  there exists a~set $A\in \mathcal{Y}^{\mathbb{Z}}$ such that
  $\nu(A)>\tau(A)=0$.  Then we would have $\nu(B)>\tau(B)=0$ for
  $B=\bigcup_{i\in\mathbb{Z}} T^iA$. But $B$ is shift-invariant, so
  $\tau(B)=\nu(B)$ by formula (\ref{FreqBarNu}). Thus, our assumption
  was false, and we rather have $\tau\gg \nu$.  Hence $\nu$ is AMS in
  view of Lemma \ref{theoDomin}.  Moreover, $\tau$ coincides with the
  expression for $\bar\nu$ given by Lemma \ref{theoBarMu}.
\end{proof}

\begin{corollary}
  \label{theoBarBarII}
  If the expansion rate (\ref{ExpRate}) is $\mu$-almost surely in the
  range $(0,\infty)$ and $\bar\mu = \bar\tau$ for two AMS measures
  $\mu$ and $\tau$, then
   \begin{align*}
     \overline{\mu\circ \okra{f^{\mathbb{Z}}}^{-1}}=
     \overline{\tau\circ \okra{f^{\mathbb{Z}}}^{-1}} 
     .
   \end{align*}
\end{corollary}
Putting $\tau=\bar\mu$, we obtain $\bar\nu=\overline{\bar\mu\circ
  \okra{f^{\mathbb{Z}}}^{-1}}$ and $\bar\nu\gg \bar\mu\circ
\okra{f^{\mathbb{Z}}}^{-1}$ for $\nu=\mu\circ
\okra{f^{\mathbb{Z}}}^{-1}$.

Under much stronger assumptions, there exists a~``finite-sum''
expression for the stationary mean $\bar\nu$ in terms of $\bar\mu$,
noticed by Kieffer and Gray for $\bar\mu=\mu$ \cite[Example
6]{GrayKieffer80}. Namely, we can construct a~stationary measure
$\rho$ by averaging the stationary mean $\bar\mu\circ
\okra{f^{\mathbb{Z}}}^{-1}$ over a~randomized shift within the
quasiperiod $\abs{f(x_1)}$. This idea is more generic, see
\cite{CariolaroPierobon77,Hurd74}.
\begin{proposition}
  \label{theoQuasiMean} 
  Let $\mu$ be an AMS measure on
  $(\mathbb{X}^{\mathbb{Z}},\mathcal{X}^{\mathbb{Z}})$ and suppose
  that the expected expansion rate (\ref{ExpExpRate}) is in the range
  $(0,\infty)$ for an $f:\mathbb{X}\rightarrow\mathbb{Y}^*$. Then
  there exists a~stationary measure
  \begin{align}
    \label{QuasiMean}
    \rho(A)=\frac{1}{L}\int G(A,{x}^{\mathbb{Z}})
    d\bar\mu({x}^{\mathbb{Z}})
    ,
    \quad
    A\in \mathcal{Y}^{\mathbb{Z}}
    .
  \end{align}
  Remark: \emph{Under the above assumptions, the expansion rate $\bar
    l(\cdot)$ may vanish on a~set of positive measure.}
\end{proposition}
\begin{proof} 
  Stationarity of $\rho$ was discussed in \cite{GrayKieffer80} for an
  injective $f^{\mathbb{Z}}$. The following proof is more general.
  First of all, we have $\rho(\mathbb{Y}^{\mathbb{Z}})=1$, whereas the
  countable additivity follows by the dominated convergence
  theorem. As for stationarity, we obtain
  \begin{align*}
    \rho(T^{-1}A)-\rho(A)
    &=
    \textstyle
    L^{-1}\int 
    \okra{F(A,0,T{x}^{\mathbb{Z}})-F(A,0,{x}^{\mathbb{Z}})}d\bar\mu ({x}^{\mathbb{Z}})=0
  \end{align*}
  from (\ref{GAshifted}) and $\bar\mu\circ T^{-1}=\bar\mu$.
\end{proof}

Although $\rho$ does not necessarily equal $\bar\nu$, it dominates the
measure $\nu$.
\begin{corollary}
  \label{theoQuasiMeanCol}
  Suppose that the hypothesis of Proposition \ref{theoQuasiMean} holds
  true and $\nu=\mu\circ \okra{f^{\mathbb{Z}}}^{-1}$ is a~measure on
  $(\mathbb{Y}^{\mathbb{Z}},\mathcal{Y}^{\mathbb{Z}})$. Then $\nu$ is
  AMS and $\rho \gg \bar\nu,\bar\mu\circ \okra{f^{\mathbb{Z}}}^{-1}
  \gg \nu$.
\end{corollary}
\begin{proof}
  Observe that $\rho(A)\ge L^{-1}\bar\mu
  (\okra{f^{\mathbb{Z}}}^{-1}A)$. Hence $\rho \gg \bar\mu\circ
  \okra{f^{\mathbb{Z}}}^{-1}$.  Since $\bar\mu\circ
  \okra{f^{\mathbb{Z}}}^{-1}\gg\mu\circ
  \okra{f^{\mathbb{Z}}}^{-1}=\nu$ follows from $\bar\mu\gg \mu$, we
  obtain $\rho\gg \nu$. Measure $\rho$ is stationary by Proposition
  \ref{theoQuasiMean}, so $\nu$ is AMS, and $\rho \gg \bar\nu$ 
  by Lemma \ref{theoDomin}.  
\end{proof}

The next proposition states that $\rho$ is the stationary mean of the
transported measure if the expansion rate is almost surely constant.
\begin{proposition}
  \label{theoBarNuRho}
  Suppose that the hypothesis of Proposition \ref{theoQuasiMean} holds
  true and $\bar l(\cdot)=L$ $\mu$-almost everywhere. Then
  $\rho=\bar\nu$ for $\nu=\mu\circ \okra{f^{\mathbb{Z}}}^{-1}$.
\end{proposition}
\begin{proof}
  By stationarity of $\bar\mu$, identity (\ref{GAshifted}), and the
  dominated convergence,
  \begin{align*}
    \rho(A)
    &=
    \textstyle
    L^{-1}\int 
    \kwad{
      \lim_{n}
      n^{-1}\sum_{i=0}^{n-1} G(A,T^i{x}^{\mathbb{Z}})
    }
    d\bar\mu ({x}^{\mathbb{Z}})
   .
  \end{align*}
  This expression equals $\bar\nu(A)$ by Proposition \ref{theoBarNu}
  if $\bar l(\cdot)=L$ almost surely.
\end{proof}

\begin{example}
  The equality $l(\cdot)=L$ holds almost everywhere for the nonergodic
  process (\ref{UDPX}) and
  $f:\mathbb{N}\times\klam{0,1}\rightarrow\klam{0,1}^*$ if
  $f(k,z)=g(k)w(z)$ for $k\in\mathbb{N}$, $z\in\klam{0,1}$,
  $\abs{g(k)}=O(\log k)$, and $\abs{w(z)}=A$. Code (\ref{ConjCode})
  falls under that case.
\end{example}

\section{Synchronization and shift-invariant $\sigma$-fields}
\label{secInvariant}

For an injection
$\pi:\mathbb{X}^{\mathbb{Z}}\rightarrow\mathbb{Y}^{\mathbb{Z}}$, the
transported shift
\begin{align}
  \label{StarT}
  T_\pi:=\pi\circ T \circ \pi^{-1}
  ,
\end{align}
considered in \cite[Example 6]{GrayKieffer80}, constitutes an
injection $\pi(\mathbb{X}^{\mathbb{Z}})\rightarrow
\pi(\mathbb{X}^{\mathbb{Z}})$. In that case, $\nu:=\mu\circ \pi^{-1}$
is stationary with respect to $T_\pi$ for a~stationary measure $\mu$
on $(\mathbb{X}^{\mathbb{Z}},\mathcal{X}^{\mathbb{Z}})$, i.e.,
$\nu\circ T_\pi^{-1}=\nu$.

We will demonstrate that ergodic properties of measures $\mu$ and
$\nu=\mu\circ \pi^{-1}$ may be further related in the more specific
case of a~synchronizable injection.  Some apparent technical
difficulty is that the set $\pi(\mathbb{X}^{\mathbb{Z}})$ usually does
not belong to the $T$-invariant algebra
$\mathcal{I}_{\mathbb{Y}}:=\klam{A\in \mathcal{Y}^{\mathbb{Z}}:
  T^{-1}A=A}$.  However, this can be overcome easily given certain
care.

\begin{lemma}
  \label{theoInvar}
  For an injection
  $\pi:\mathbb{X}^{\mathbb{Z}}\rightarrow\mathbb{Y}^{\mathbb{Z}}$,
  consider pseudo-invariant algebras
  \begin{align*}
    \mathcal{Q}&:=\klam{A\in \mathcal{Y}^{\mathbb{Z}}: A=B\cap
      \pi(\mathbb{X}^{\mathbb{Z}}),\, T^{-1}B=B}
    ,
    \\
    \mathcal{Q}_\pi&:=\klam{A\in \mathcal{Y}^{\mathbb{Z}}: A=B\cap
      \pi(\mathbb{X}^{\mathbb{Z}}),\, {T_\pi}^{-1}B=B}
    ,
  \end{align*}
  where $T_\pi$ is defined by (\ref{StarT}). We have
  \begin{align*}
   \mathcal{Q}
   \subset \mathcal{Q}_\pi
   =\pi(\mathcal{I}_{\mathbb{X}})
    .
  \end{align*}
\end{lemma}
\begin{proof}
  The right equality is obvious. As for the left relation, observe
  that $\bigcup_{i\in\mathbb{Z}} T^i B\supset \bigcup_{i\in\mathbb{Z}}
  T_\pi^i B \supset B$.  If $T^{-1}B=B$ then $\bigcup_{i\in\mathbb{Z}}
  T^iB=B$. Hence $B\cap \pi(\mathbb{X}^{\mathbb{Z}})\in \mathcal{Q}_\pi$
  since formula $\bigcup_{i\in\mathbb{Z}} T_\pi^i B$ defines
  a~$T_\pi$-invariant set.
\end{proof}


\begin{proposition}
  \label{theoSync}
  For a~synchronizable injection
  $\pi:\mathbb{X}^{\mathbb{Z}}\rightarrow\mathbb{Y}^{\mathbb{Z}}$,
  \begin{align*}
    \mathcal{Q} = \mathcal{Q}_\pi
    .
  \end{align*}
\end{proposition}
\begin{proof}
  By Lemma \ref{theoInvar}, $\mathcal{Q} \subset \mathcal{Q}_\pi$.
  Thus it suffices to show that $\mathcal{Q}_\pi\subset \mathcal{Q}$
  or, equivalently, that $\mathcal{I}_{\mathbb{X}}\subset
  \pi^{-1}(\mathcal{Q})$.  We will demonstrate the latter.  Consider
  an $A\in\mathcal{I}_{\mathbb{X}}$ and construct the set
  $E=\pi(\mathbb{X}^{\mathbb{Z}})\cap \bigcup_{i\in\mathbb{Z}} T^i
  \pi(A)\in \mathcal{Q}$. Since $\pi$ is synchronizable and $A$ is
  $T$-invariant, we have that $\pi^{-1}(E)=A$.
\end{proof}

\begin{proposition}
  \label{theoSyncII}
  Consider a~synchronizable injection
  $\pi:\mathbb{X}^{\mathbb{Z}}\rightarrow\mathbb{Y}^{\mathbb{Z}}$,
  a~measure $\mu$ on
  $(\mathbb{X}^{\mathbb{Z}},\mathcal{X}^{\mathbb{Z}})$, and its image
  $\nu=\mu\circ\pi^{-1}$ on
  $(\mathbb{Y}^{\mathbb{Z}},\mathcal{Y}^{\mathbb{Z}})$.  For each
  $E\in\mathcal{I}_{\mathbb{Y}}$, there exists such an $A\in
  \mathcal{I}_{\mathbb{X}}$, and for each $A\in
  \mathcal{I}_{\mathbb{X}}$, there exists an
  $E\in\mathcal{I}_{\mathbb{Y}}$ such that
  \begin{align}
    \nu(E)=\nu(E\cap \pi(\mathbb{X}^{\mathbb{Z}}))
    =\mu(A)
    \label{InvarProb}
    .
  \end{align}
  Remark: \emph{As a~further corollary, either both measures $\mu$ and
    $\nu$ are ergodic or neither of them has this property.  Let us
    recall that $\bar\mu(A)=\mu(A)$ for $A\in
    \mathcal{I}_{\mathbb{X}}$ and an AMS $\mu$. The analogical
    equality $\bar\nu(E)=\nu(E)$ holds for $E \in
    \mathcal{I}_{\mathbb{Y}}$ and an AMS $\nu$. Some oddness of
    (\ref{InvarProb}) is buried in the fact that $\bar\nu(E)$ does not
    necessarily equal $\bar\nu(E\cap \pi(\mathbb{X}^{\mathbb{Z}}))$.
    It is only the support of $\nu$ that is confined to
    $\pi(\mathbb{X}^{\mathbb{Z}})$, and $\pi(\mathbb{X}^{\mathbb{Z}})$
    need not be $T$-invariant, as it has been remarked.}
\end{proposition}
\begin{proof}
  For $E\in\mathcal{I}_{\mathbb{Y}}$, take $A=\pi^{-1}(E)$. For $A\in
  \mathcal{I}_{\mathbb{X}}$, take $E=\bigcup_{i\in\mathbb{Z}} T^i
  \pi(A)$. Then the equality follows immediately from Proposition
  \ref{theoSync}.
\end{proof}

\begin{example}
  Process (\ref{UDPX}) has a~nonatomic shift-invariant
  sub-$\sigma$-field \cite{Debowski09}. Hence the expanded process
  $(Y_i)_{i\in\mathbb{Z}}=f^{\mathbb{Z}}((X_i)_{i\in\mathbb{Z}})$
  distributed $P((Y_i)_{i\in\mathbb{Z}}\in\cdot)=\nu$ and its
  stationary mean $(\bar Y_i)_{i\in\mathbb{Z}}$ distributed $\bar
  P((\bar Y_i)_{i\in\mathbb{Z}}\in\cdot)=\overline{\nu}$ have the same
  property if we use a~comma-separated code
  $f:\mathbb{X}\rightarrow\mathbb{Y}^*$, like (\ref{ConjCode}).
  Moreover, $(\bar Y_i)_{i\in\mathbb{Z}}$ has a~nonatomic
  shift-invariant sub-$\sigma$-field if and only if (\ref{UDPcondY})
  holds true for certain functions $\bar s_{k}$ and independent
  equidistributed binary random variables $(\bar
  Z_k)_{k\in\mathbb{N}}$ \cite[Theorem 9]{Debowski09}.
\end{example}

\section{Complete fix-free codes and stationarity}
\label{secShrunk}

This section contains a~result of independent interest, loosely
related to the setting of our initial problem.  For any injection
$\pi:\mathbb{X}^{\mathbb{Z}}\rightarrow\mathbb{Y}^{\mathbb{Z}}$,
measures may also be transported in the opposite direction. That is,
for any measure $\nu$ on
$(\mathbb{Y}^{\mathbb{Z}},\mathcal{Y}^{\mathbb{Z}})$, there exists
a~measure $\mu:=\nu\circ {\pi}$ on
$(\mathbb{X}^{\mathbb{Z}},\mathcal{X}^{\mathbb{Z}})$. A~condition
opposite to synchronization appears when this mapping is required to
preserve stationarity.  We came across the following proposition,
which seemingly has not been noticed so far, cf.\
\cite{BajiStefanovicVukobratovic05}:
\begin{proposition} 
  \label{theoShrunkF}
  Suppose that $\mathbb{X}$ is finite and $f:\mathbb{X}\rightarrow\mathbb{Y}^*$
  is complete fix-free. Then
  \begin{enumerate}
  \item $f^{\mathbb{Z}}$ is a~bijection
    $\mathbb{X}^{\mathbb{Z}}\rightarrow\mathbb{Y}^{\mathbb{Z}}$, and
  \item $\mu=\nu\circ {f^{\mathbb{Z}}}$ is a~stationary measure on
    $(\mathbb{X}^{\mathbb{Z}},\mathcal{X}^{\mathbb{Z}})$ if $\nu$ is
    a~stationary measure on
    $(\mathbb{Y}^{\mathbb{Z}},\mathcal{Y}^{\mathbb{Z}})$.
  \end{enumerate}
  Remark: \emph{Statement (i) may be false for a~complete infinite
    prefix-free set $f(\mathbb{X})$. For instance, the set
    $\okra{f^{\mathbb{Z}}}^{-1}(\klam{...00\bm{.}00..})$ is empty for
    $f(n)=\klam{0^{n-1}1: n\in\mathbb{N}}$, $\mathbb{X}=\mathbb{N}$,
    and $\mathbb{Y}=\klam{0,1}$.  However, we do not know of any
    complete infinite set of strings that would be both prefix- and
    suffix-free, cf.\
    \cite{GillmanRivest95,AhlswedeBalkenholKhachatrian96}.}
\end{proposition}
\begin{proof}
  Let $\mathcal{L}=f(\mathbb{X})$.

  (i) Clearly, $|w|\ge 1$ for $w\in\mathcal{L}$ if $f$ is a~bijection.
  Thus $f^{\mathbb{Z}}({x}^{\mathbb{Z}})$ is a~two-sided sequence for
  ${x}^{\mathbb{Z}}\in\mathbb{X}^{\mathbb{Z}}$. Moreover, given
  a~${y}^{\mathbb{Z}}\in\mathbb{Y}^{\mathbb{Z}}$, we can reconstruct
  the unique ${x}^{\mathbb{Z}}\in\mathbb{X}^{\mathbb{Z}}$ with
  $f^{\mathbb{Z}}({x}^{\mathbb{Z}})={y}^{\mathbb{Z}}$ by cutting off
  the consecutive suffixes or prefixes belonging to $\mathcal{L}$ from
  $y_{-\infty}^{0}$ and $y_{1}^{\infty}$.  By the following reasoning,
  this parsing process is guaranteed not to stop after a~finite number
  of steps.

  On the contrary, assume that there is such an infinite sequence
  $y_{k}^{\infty}$ (the mirrorlike argumentation applies to
  $y_{-\infty}^{k}$) such that no $w\in \mathcal{L}$ is a~prefix of
  $y_{k}^{\infty}$. Let $v$ be a~prefix of $y_{k}^{\infty}$ that is
  longer than any $w\in \mathcal{L}$. The set $\klam{v}\cup
  \mathcal{L}$ is prefix-free, and so $\sum_{w\in \mathcal{L}}
  \abs{\mathbb{Y}}^{-\abs{w}}\le 1-\abs{\mathbb{Y}}^{-\abs{v}}<1$ by
  the Kraft inequality.  We have arrived at a~contradiction, so the
  assumption was false.
  
  (ii) By the Kolmogorov process theorem and the
  $\pi$-$\lambda$ theorem, stationarity of $\mu$ is equivalent
  to the set of equalities 
  \begin{align}
    \label{MuSta}
    \sum_{z\in \mathcal{L}} \nu([wz])=\nu([w])=\sum_{z\in \mathcal{L}}
    \nu([zw]), \quad w\in \mathcal{L}^*
    .
  \end{align}
  On the other hand, stationarity of $\nu$ is equivalent to 
  \begin{align}
    \label{NuSta}
    \sum_{s\in\mathbb{Y}} \nu([ws])=\nu([w])=\sum_{s\in\mathbb{Y}}
    \nu([sw]), \quad w\in \mathbb{Y}^*
    .
  \end{align}
  
  The following auxiliary fact is useful to derive (\ref{MuSta}) from
  (\ref{NuSta}): Let $l(\mathcal{M})\ge 1$ be the length of the
  longest string in a~set $\mathcal{M}$. Any finite complete prefix-free
  set $\mathcal{M}\subset\mathbb{Y}^*$ may be decomposed as
  $\mathcal{M}=\mathcal{M}_r\cup (\mathcal{M}_p\times\mathbb{Y})$,
  where
  $l(\mathcal{M})=l(\mathcal{M}_p\times\mathbb{Y})>l(\mathcal{M}_r)$,
  and $\mathcal{M}_m=\mathcal{M}_r\cup \mathcal{M}_p$ is a~complete
  prefix-free set. This decomposition may be proved by contradiction
  with the Kraft inequality applied to $\mathcal{M}_m$.

  Fron this and from (\ref{NuSta}), it follows that for any complete
  prefix-free $\mathcal{M}$ with $l(\mathcal{M})\ge 1$, there exists
  a~complete prefix-free $\mathcal{M}_m$ such that
  $l(\mathcal{M})=l(\mathcal{M}_m)-1$ and
  \begin{align*}
    \sum_{z\in \mathcal{M}} \nu([wz])=\sum_{z\in \mathcal{M}_m} \nu([wz])
    .
  \end{align*}
  Using this, the left equality in (\ref{MuSta}) may be proved by
  induction on $l(\mathcal{M})$ starting with
  $\mathcal{M}_m=\klam{\lambda}$. The proof of the right equality is
  mirrorlike.
\end{proof}

\begin{corollary}
  \label{theoShrunkFAMS}
  Suppose that $\mathbb{X}$ is finite,
  $f:\mathbb{X}\rightarrow\mathbb{Y}^*$ is complete fix-free, and
  $\nu$ is AMS on
  $(\mathbb{Y}^{\mathbb{Z}},\mathcal{Y}^{\mathbb{Z}})$.  Then
  $\mu=\nu\circ {f^{\mathbb{Z}}}$ is AMS on
  $(\mathbb{X}^{\mathbb{Z}},\mathcal{X}^{\mathbb{Z}})$ with the
  stationary mean $\bar\mu\ll \bar\nu\circ {f^{\mathbb{Z}}}$.
\end{corollary}
\begin{proof}
  We have $\mu=\nu\circ {f^{\mathbb{Z}}}\ll \bar\nu\circ
  {f^{\mathbb{Z}}}$, where the last measure is stationary by Proposition
  \ref{theoShrunkF}. Hence the claim follows by Lemma \ref{theoDomin}.
\end{proof}

\section{Preservation of the finite-energy property}
\label{secEnergy}

We supposed that both $f^{\mathbb{Z}}$ and
$\okra{f^{\mathbb{Z}}}^{-1}$ preserve the finite-energy property if
the coding function $f$ is sufficiently nice, prefix-free in
particular. The proofs are a~bit more complicated than we expected,
but convenient sufficient conditions can be formulated.

\begin{definition}
  More specifically, we will say that (i) a~measure $\mu$ on
  $(\mathbb{X}^{\mathbb{Z}},\mathcal{X}^{\mathbb{Z}})$ \emph{has
    $(K,c)$-energy} if $c<1$, $K<\infty$, and condition (\ref{FE})
  holds and (ii) the measure $\mu$ \emph{has $(K,c,f)$-energy} for
  a~coding function $f:\mathbb{X}\rightarrow\mathbb{Y}^*$ if $c<1$,
  $K<\infty$, and
  \begin{align}
    \label{gFE}
    \mu([uv])\le Kc^{\abs{f^*(v)}} \mu([u])
    .
  \end{align}
  Remark: \emph{If a~function $f:\mathbb{X}\rightarrow\mathbb{Y}^*$ is
    prefix-free, then, by the Kraft inequality $\sum_{x\in\mathbb{X}}
    \abs{\mathbb{Y}}^{-\abs{f(x)}}\le 1$, condition (\ref{gFE}) may be
    only satisfied for $c\ge \abs{\mathbb{Y}}^{-1}$. In particular,
    the inequality $c> \abs{\mathbb{Y}}^{-1}$ must be strict for
    a~noncomplete coding function, i.e., when $\sum_{x\in\mathbb{X}}
    \abs{\mathbb{Y}}^{-\abs{f(x)}}< 1$.}
\end{definition}

\begin{proposition}
  \label{theoFENuFEMu}
  If $f:\mathbb{X}\rightarrow\mathbb{Y}^*$ is prefix-free and
  a~measure $\mu$ on
  $(\mathbb{X}^{\mathbb{Z}},\mathcal{X}^{\mathbb{Z}})$ has
  $(K,c,f)$-energy, then $\mu$ has also $(K,c)$-energy.
\end{proposition}
\begin{proof}
  If $f$ is prefix-free, then $\abs{f^*(u)}\ge \abs{u}$. Hence,
\begin{align*}
  \mu([vu]) \le Kc^{\abs{f^*(u)}} \mu([v])
  \le Kc^{\abs{u}} \mu([v])
  .
\end{align*}
\end{proof}

\begin{proposition}
  \label{theoFENugFEMu}
  If $f:\mathbb{X}\rightarrow\mathbb{Y}^*$ is prefix-free and
  a~measure $\nu=\mu\circ \okra{f^{\mathbb{Z}}}^{-1}$ on
  $(\mathbb{Y}^{\mathbb{Z}},\mathcal{Y}^{\mathbb{Z}})$ has
  $(K,c)$-energy, then the measure $\mu$ on
  $(\mathbb{X}^{\mathbb{Z}},\mathcal{X}^{\mathbb{Z}})$ has
  $(K,c,f)$-energy.
\end{proposition}
\begin{proof}
  Let $\nu([zw])\le Kc^{\abs{w}} \nu([z])$. The extension $f^*$ is an
  injection for a~prefix-free $f$. Moreover, if both $z$ and $w$
  belong to $f^*(\mathbb{X}^*)$, then
  $$\okra{f^{\mathbb{Z}}}^{-1}([zw])=[\okra{f^*}^{-1}(z)\okra{f^*}^{-1}(w)].$$ 
  In particular, for $z=f^*(u)$ and $w=f^*(v)$, we obtain
\begin{align*}
  \mu([uv])
  =\nu([zw])
  &
  \le Kc^{\abs{w}} \nu([z])
  = Kc^{\abs{f^*(v)}} \mu([u])
  .
\end{align*}
\end{proof}

The converse of Proposition \ref{theoFENugFEMu} is valid under additional
restrictions.
%
%
%
Denote the difference of strings $w,z\in\mathbb{Y}^*$ as
\begin{align*}
  w\stringminus z:=
  \begin{cases}
    s & \text{if } w=zs, s\in\mathbb{Y}^+,
    \\
    \lambda & \text{if } z=ws, s\in\mathbb{Y}^*,
    \\
    w & \text{else}.
  \end{cases}
\end{align*}
We choose this definition to have $p^{\abs{w\stringminus z}}\le
p^{\abs{w}-\abs{z}}$ for $p\in(0,1)$.  The set of $z$'s that fall
under the first two cases is denoted as
\begin{align*}
  \dzi{w}:=
  \klam{z\in\mathbb{Y}^*: \exists_{s\in\mathbb{Y}^*}: w=zs \lor z=ws}.
\end{align*}
Moreover, for a~set $\mathcal{L}\subset\mathbb{Y}^*$, we define the
remainder $w_{\mathcal{L}}$ of a~string $w\in\mathbb{Y}^*$ as the
shortest element of the set $\klam{w\stringminus s:s\in \mathcal{L}^*,
  \abs{s}\le \abs{w}}$. The completion set of the string $w$ with
respect to the set $\mathcal{L}$ is defined as
\begin{align*}
  \mathcal{L}_w:=
    \klam{s\in\mathbb{Y}^*: 
      w_{\mathcal{L}}s \in\mathcal{L}}.
\end{align*}
This set is prefix-free for a~prefix-free $\mathcal{L}$.

\begin{proposition}
  \label{theogFEMuFENu}
  Suppose that $f:\mathbb{X}\rightarrow\mathbb{Y}^*$ is prefix-free
  and a~measure $\mu$ on
  $(\mathbb{X}^{\mathbb{Z}},\mathcal{X}^{\mathbb{Z}})$ has
  $(K,c,f)$-energy. For $\mathcal{L}=f(\mathbb{X})$, put
  \begin{align*}
    M_f(p)&:=\sup_{w\in\mathbb{Y}^*:\mathcal{L}_w\not=\emptyset} \sum_{s\in \mathcal{L}_w}
    p^{\abs{s}} ,
    \\
    N_{f,\mu}(p)&:=
    \sup_{w\in \mathbb{Y}^*:\mathcal{L}_w\not=\emptyset} 
    \frac{
      \sum_{s\in\mathcal{L}_w} p^{-\abs{s}} \mu([\okra{f^*}^{-1}(ws)])
    }{
      \sum_{s\in\mathcal{L}_w} \mu([\okra{f^*}^{-1}(ws)])
    }
    .
  \end{align*}
  If $M_f(c)<\infty$ and $N_{f,\mu}(c_2)<\infty$ for a~certain $c_2\in
  [c,1)$, then the measure $\nu=\mu\circ \okra{f^{\mathbb{Z}}}^{-1}$
  on $(\mathbb{Y}^{\mathbb{Z}},\mathcal{Y}^{\mathbb{Z}})$ has $(\tilde
  K, c_2)$-energy, where $\tilde K=N_{f,\mu}(c_2) M_f(c) K$.
\end{proposition}
\begin{proof}
  If $z=f^*(u)$ and $w=f^*(v)$ for some $u$ and $v$, then 
\begin{align*}
  \nu([zw])
  = \mu([uv])
  &
  \le Kc^{\abs{f^*(v)}} \mu([u])
  = Kc^{\abs{w}} \nu([z])
  .
\end{align*}
Notice that $\nu([z])=\sum_{s\in \mathcal{L}_z} \nu([zs])$ for any
$z\in\mathbb{Y}^*$.  Assume now that $z\in \mathcal{L}^*$ and let $w$
be arbitrary. By $\mathcal{L}_{zw}=\mathcal{L}_w$ we obtain
\begin{align*}
  \nu([zw])
  &
  =\sum_{s\in\mathcal{L}_w} \nu([zws])
  \le \sum_{s\in\mathcal{L}_w} Kc^{\abs{ws}} \nu([z])
  \le M_f(c) Kc^{\abs{w}} \nu([z])
  .
\end{align*}
Eventually, consider arbitrary $z$ and $w$.   We have
\begin{align*}
  \nu([zw]) 
  &=
  \sum_{s\in\mathcal{L}_z\cap \dzi{w}} \nu([zs(w\stringminus s)]) 
  \le
  \sum_{s\in\mathcal{L}_z\cap \dzi{w}} M_f(c) K c^{\abs{w\stringminus s}} \nu([zs])
  \\
  &\le 
  \sum_{s\in\mathcal{L}_z} M_f(c) K c_2^{\abs{w\stringminus s}} \nu([zs]) 
  \le 
  M_f(c) K c_2^{\abs{w}} \sum_{s\in\mathcal{L}_z} c_2^{-\abs{s}} \nu([zs])
  \\
  &\le \tilde K c_2^{\abs{w}} \nu([z]) 
  .
\end{align*}
\end{proof}

\begin{corollary}
  \label{theogFEMuFENuFin}
  If $f:\mathbb{X}\rightarrow\mathbb{Y}^*$ is finite prefix-free and
  a~measure $\mu$ on
  $(\mathbb{X}^{\mathbb{Z}},\mathcal{X}^{\mathbb{Z}})$ has finite
  energy, then the measure $\nu=\mu\circ \okra{f^{\mathbb{Z}}}^{-1}$ on
  $(\mathbb{Y}^{\mathbb{Z}},\mathcal{Y}^{\mathbb{Z}})$ has finite
  energy.
\end{corollary}
\begin{proof}
  Assume that $\mu$ has $(K,q)$-energy. We have
  $\abs{\mathcal{L}_w}\le \abs{\mathcal{L}}<\infty$ and
  $\sup_{z\in\mathcal{L}_w}\abs{z}\le
  \sup_{z\in\mathcal{L}}\abs{z}<\infty$. Hence we obtain inequalities
  $M_f(p)<\infty$, $N_{f,\mu}(p)<\infty$, and (\ref{gFE}) for $p<1$
  and $c\in [\max_{z\in\mathcal{L}} q^{1/\abs{z}},1)$. In consequence,
  the claim follows by Proposition \ref{theogFEMuFENu}.
\end{proof}

Below we present a~more specific example with an infinite image
$f(\mathbb{X})$.
\begin{corollary}
  \label{theogFEMuFENuInf}
  Let $f:\mathbb{X}=\mathbb{N}\mapsto
  \mathbb{Y}^+=\klam{0,1,2}^+$ be given as
  \begin{align*}
    f(k)=b(k)w(k)2
    ,
  \end{align*}
  where $1b(k)\in\klam{0,1}^+$ is the binary representation of
  a~natural number $k$, and $w(k)\in\klam{0,1}^*$ is a~string of fixed
  length, $\abs{w(k)}=A$.  Let also $\mu$ be a~measure on
  $(\mathbb{X}^{\mathbb{Z}},\mathcal{X}^{\mathbb{Z}})$ that satisfies
  $\mu([vu])=\mu([v])\mu([u])$ and
  \begin{align*}
    \mu([k])= \frac{k^{-\alpha}}{\zeta(\alpha)}, 
    \quad k\in\mathbb{N},
  \end{align*}
  for some $\alpha>1$.  If $\zeta(\alpha)> 2^{A+1}$, then $\mu$
  has $(1, c, f)$-energy for $c\in
  [\max\klam{2^{-\alpha},(\zeta(\alpha))^{-1/(A+1)}},2^{-1})$, whereas
  $\nu=\mu\circ \okra{f^{\mathbb{Z}}}^{-1}$ has $(\tilde K,
  c_2)$-energy for $\tilde K=N_{f,\mu}(c_2) M_f(c)$ and
  $c_2\in(\max\klam{c,2^{1-\alpha}},1)$.
\end{corollary}
\begin{proof}
  We have $\abs{f(k)}=\floor{\log_2 k}+1+\abs{w(k)}$ and
  $k^{-\alpha}=(2^{-\alpha})^{\log_2 k}$.
  Thus,
  \begin{align*}
     (\zeta(\alpha))^{-\abs{u}} (2^{-\alpha})^{\abs{f^*(u)}-A\abs{u}} 
    \le
    \mu([u])
    &
    \le (\zeta(\alpha))^{-\abs{u}} (2^{-\alpha})^{\abs{f^*(u)}-(A+1)\abs{u}}
    \le c^{\abs{f^*(u)}}
    .
  \end{align*}
  In particular, (\ref{gFE}) follows for $K=1$. Consider a~string
  $w\in \mathbb{Y}^*$ and let $a_l$ be the number of strings of length
  $l$ in the set $\mathcal{L}_w$. We can see that $a_l\le 1$ for $l\le
  A+1$, whereas $a_l= 2^{l-(A+1)}$ for $l> A+1$ if $\mathcal{L}_w$ is
  not empty. Hence, $M_f(c)
  \le \sum_{l=0}^\infty \max\klam{1,2^{l-(A+1)}} c^{l}<\infty$ and
  \begin{align*}
    N_{f,\mu}(c_2)
    \le
    \frac{
      \sum_{l=0}^\infty \max\klam{1,2^{l-(A+1)}} c_2^{-l} (2^{-\alpha})^{l-1}
    }{
      \sum_{l=A+1}^\infty 2^{l-(A+1)} (2^{-\alpha})^{l}
    }    
    <\infty
    .
  \end{align*}
  So the claim holds by Proposition \ref{theogFEMuFENu}.
\end{proof}

By means of the following two simple statements, the above result can
be extended to certain nonergodic measures, including the distribution
of process (\ref{UDPX}) and its stationary variable-length coding.
\begin{proposition}
  \label{theoFENuInt}
  Consider a~measure $P$ on $(\Omega,\mathcal{J})$ and a~probability
  kernel $\tau$ from $(\Omega,\mathcal{J})$ to
  $(\mathbb{Y}^{\mathbb{Z}},\mathcal{Y}^{\mathbb{Z}})$ (i.e.,
  $\tau(\cdot,\omega)$ is a~measure on
  $(\mathbb{Y}^{\mathbb{Z}},\mathcal{Y}^{\mathbb{Z}})$ for $P$-almost
  all $\omega\in\Omega$, and the function $\tau(A,\cdot)$ is measurable
  $\mathcal{J}$ for each $A\in\mathcal{Y}^{\mathbb{Z}}$). If
  $\tau(\cdot,\omega)$ has $(K,c)$-energy for $P$-almost all
  $\omega\in\Omega$, then so does the measure $\int
  \tau(\cdot,\omega)dP(\omega)$.
\end{proposition}
\begin{proof}
\begin{align*}
  \textstyle
  \int \tau([zw],\omega)dP(\omega)
  &
  \textstyle
  \le 
  \int Kc^{\abs{w}} \tau([w],\omega)dP(\omega) 
  \le 
  Kc^{\abs{w}} \int \tau([w],\omega)dP(\omega)
  .
\end{align*}  
\end{proof}

\begin{proposition}
  \label{theoFENuBar}
  If an AMS measure $\nu$ on
  $(\mathbb{Y}^{\mathbb{Z}},\mathcal{Y}^{\mathbb{Z}})$ has
  $(K,c)$-energy, then so does the measure $\bar\nu$.
\end{proposition}
\begin{proof}
\begin{align*}
  \bar\mu([zw])
  &
  = \textstyle
  \lim_n n^{-1}\sum_{i=0}^{n-1}\sum_{s\in\mathbb{Y}^i}\mu([szw])
  \\
  &
  \le \textstyle
  Kc^{\abs{w}} 
  \lim_n n^{-1}\sum_{i=0}^{n-1}\sum_{s\in\mathbb{Y}^i}\mu([sz])
  \le 
  Kc^{\abs{w}} \bar\mu([z])
  .
\end{align*}  
\end{proof}

\section{Block entropies of stationary means}
\label{secEntropy}

For a~stationary measure $\mu$, block entropy $H_\mu(n)=H_\mu(0;n)$
defined in (\ref{BlockEntropy}) is a~nonnegative, growing, and concave
function of $n$, see \cite{CrutchfieldFeldman01}. Hence the limit
\begin{align}
  \label{hBarMu}
   h_\mu:=\lim_{n\rightarrow\infty}\frac{H_\mu(n)}{n}
   ,
\end{align}
known as the entropy rate, exists in that case.  Whereas block entropy
behaves less regularly in a~general AMS case, we can bound the block
entropy of the stationary mean in the following way.
\begin{proposition}
  \label{theoHBarMuHMu}
  For an AMS measure $\mu$,
  \begin{align}
   \label{HBarMuHMu}
   H_{\bar\mu}(m)\ge\limsup_{n\rightarrow\infty} 
    \frac{1}{n}\sum_{i=0}^{n-1}H_\mu(i;m)
    .
  \end{align}
\end{proposition}
\begin{proof}
  The claim
  \begin{align*}
    - \sum_{u\in\mathbb{X}^m} 
    \bar\mu([u]) \log \bar\mu([u])
    &\ge
    \limsup_{n\rightarrow\infty}  
    \kwad{
      -  \frac{1}{n} \sum_{i=0}^{n-1} 
      \sum_{u\in\mathbb{X}^m} 
      \mu(T^{-i}[u]) \log \mu(T^{-i}[u])
    }
  \end{align*}
  follows by the Jensen inequality for the function $p\mapsto -p\log
  p$.
\end{proof}
\begin{proposition}
  \label{theohBarMuhMu}
  Let $\mu$ be AMS with $h_{\bar\mu}<\infty$ and $H_\mu(n)<\infty$ for
  all $n$. Then
  \begin{align}
    \label{hBarMuhMu}
    h_{\bar\mu} \le \liminf_{n\rightarrow\infty} \frac{H_\mu(n)}{n}
    .
  \end{align}
\end{proposition}
\begin{proof}
  By the generalized Shannon-McMillan-Breiman theorem \cite[Theorem
  2]{Barron85} and (\ref{EqInvar}), the $L^1(\bar\mu)$ convergence
  \begin{align}
    \label{BarMuAEP}
    h_{\bar\mu}
    =
    - \int \lim_{n\rightarrow\infty} \frac{\log \bar\mu([x^{n}])}{n}
    d\bar\mu({x}^{\mathbb{Z}})
    =
    - \int \lim_{n\rightarrow\infty} \frac{\log \bar\mu([x^{n}])}{n} 
    d\mu({x}^{\mathbb{Z}})
  \end{align}
  holds for the stationary measure $\bar\mu$ if
  $h_{\bar\mu}<\infty$. On the other hand, by \cite[Theorem
  3]{Barron85}, an AMS measure $\mu$ with $H_\mu(n)<\infty$ satisfies
  \begin{align}
    \label{MuAEP}
    \lim_{n\rightarrow\infty} \frac{\log \bar\mu([x^{n}])}{n}
    =
    \lim_{n\rightarrow\infty} \frac{\log \mu([x^{n}])}{n}
  \end{align}
  for $\mu$- and $\bar\mu$-almost all ${x}^{\mathbb{Z}}$. Hence the
  claim follows by the Fatou lemma.
%
\end{proof}

\begin{example}
  \label{exChampernowne}
  Using Proposition \ref{theohBarMuhMu}, we can show a~simple example
  of a~measure $\mu$ such that limits (\ref{BarMu}) exist for all
  cylinder sets but $\mu$ is not AMS. Consider the Champernowne
  sequence $b^{\mathbb{N}}=12345678910111213...$ (i.e., the
  concatenation of decimal representations of natural numbers) and put
  $\mu([b^k])=1$ for $k\in\mathbb{N}$. We have
  $\bar\mu([u])=10^{-|u|}$ since $b^{\mathbb{N}}$ is normal. If
  $\bar\mu$ is extended from these values to a~measure on
  $(\mathbb{X}^{\mathbb{Z}},\mathcal{X}^{\mathbb{Z}})$, then
  $H_{\bar\mu}(m)=m\log 10$ and $h_{\bar\mu}=\log 10$ but
  $H_\mu(i;m)=0$. Hence $\mu$ cannot be AMS.
\end{example}

Block entropies of two stationary means linked through variable-length
coding can be related as well.  The link for the entropy rate is very
simple if the expansion rate is constant and the coding function is
uniquely decodable.
\begin{proposition}[\mbox{cf.\ \cite[Theorem 1]{TimoBlackmoreHanlen2007}}]
  \label{theohBarMuhBarNu}
  Let $\mu$ be AMS on
  $(\mathbb{X}^{\mathbb{Z}},\mathcal{X}^{\mathbb{Z}})$ with
  $h_{\bar\mu}<\infty$ and $H_\mu(n)<\infty$ for all $n$ and suppose
  that the expansion rate satisfies $\bar l(\cdot)=L\in (0,\infty)$
  $\mu$-almost everywhere for a~prefix-free
  $f:\mathbb{X}\rightarrow\mathbb{Y}^*$. Then we have
  \begin{align}
    \label{hBarMuhBarNu}
    h_{\bar\nu}=L^{-1}h_{\bar\mu}
  \end{align}
  for the measure $\nu=\mu\circ \okra{f^{\mathbb{Z}}}^{-1}$ if
  $h_{\bar\nu}<\infty$.
  \\
  Remark: \emph{We have $h_{\bar\mu}<\infty$ and $H_\mu(n)<\infty$ if
    $\mu$ is stationary and $H_\mu(1)<\infty$, whereas
    $h_{\bar\nu}<\infty$ if the alphabet $\mathbb{Y}$ is finite.
    Formula (\ref{hBarMuhBarNu}) is a~special case of \cite[Theorem
    1]{TimoBlackmoreHanlen2007}, but their proof is partly flawed. It
    uses the version of the Shannon-McMillan-Breiman theorem
    \cite[Corollary 4]{GrayKieffer80}, which is true only for finite
    $\mathbb{X}$ and $\mathbb{Y}$. A~correct proof for any
    $\mathbb{X}$ and $\mathbb{Y}$, invoking the already mentioned
    Theorems 2 and 3 from \cite{Barron85}, is given below. As noticed
    in \cite[Theorem 1]{TimoBlackmoreHanlen2007}, $\lim_n n^{-1}\log
    \mu([x^{n}])$ converges almost surely to the entropy rate of the
    $x^{\mathbb{Z}}$-typical ergodic component of the measure $\mu$
    also when the expansion rate is not constant, see the ergodic
    decomposition theorems in \cite[Chapter 9 until Theorem
    9.12]{Kallenberg97}.}
\end{proposition}
\begin{proof}
  The measure $\nu=\mu\circ \okra{f^{\mathbb{Z}}}^{-1}$ is AMS by Proposition
  \ref{theoBarNu}, whereas the extension $f^*$ is an injection for
  a~prefix-free $f$. Hence $\mu([x^{n}])=\nu([f^*(x^{n})])$
  and $\abs{f^*(x)}\ge\abs{x}$. As a~result, $H_\nu(n)\le H_\mu(n)$,
  whereas (\ref{BarMuAEP}), (\ref{MuAEP}), and their analogues for
  $\nu$ imply
 \begin{align*}
   h_{\bar\nu}
   &=
   - \int 
   \kwad{\lim_{n\rightarrow\infty} n^{-1}\log \nu([y^{n}])}
   d\nu({y}^{\mathbb{Z}})
   \\
   &=
   - \int 
   \kwad{\lim_{n\rightarrow\infty} \abs{f^*(x^{n})}^{-1}\log \mu([x^{n}])}
   d\mu({x}^{\mathbb{Z}})
   \\
   &=
   - \int
   \kwad{l({x}^{\mathbb{Z}})}^{-1}
   \kwad{\lim_{n\rightarrow\infty} n^{-1}\log \mu([x^{n}])}
   d\mu({x}^{\mathbb{Z}})
   =
   L^{-1}h_{\bar\mu}
   .
  \end{align*}
\end{proof}

In the following, we wish to obtain some bounds for $H_{\bar\nu}(n)$
in terms of $H_{\bar\mu}(n)$ for finite $n$. We shall observe that
formula (\ref{QuasiMean}) can be interpreted in terms of random
variables if $\abs{f(x)}>0$ for all $x\in\mathbb{X}$. To simplify the
notation, we shall assume here that
$\bar\mu=\mu:=P((X_i)_{i\in\mathbb{Z}}\in\cdot)$ is the measure of
a~stationary process $(X_i)_{i\in\mathbb{Z}}$. Then $\rho=\bar P((\bar
Y_i)_{i\in\mathbb{Z}}\in\cdot)$ is the stationary distribution of
\begin{align}
  \label{DefBarY}
  (\bar Y_i)_{i\in\mathbb{Z}}=T^{N} f^{\mathbb{Z}}((\bar X_i)_{i\in\mathbb{Z}})
  ,
\end{align}
where the random shift $N$ and the nonstationary process $(\bar
X_i)_{i\in\mathbb{Z}}$ are conditionally independent given $\bar X_1$,
their distribution being
\begin{align}
  \label{DefBarX}
  \bar P(\bar X_{k}^{l}=x_{k}^{l})
  &= 
  P(X_{k}^{l}=x_{k}^{l})\cdot \frac{\abs{f(x_1)}}{L} 
  ,
  &
  k&\le 1\le l,
  \\
  \label{DefN}
  \bar P(N=n|\bar X_{1}=x_{1})
  &=
  \frac{
    \boole{0\le n\le \abs{f(x_1)}-1}
  }{
    \abs{f(x_1)}
  }
  ,
  &
  n&\in\mathbb{N}\cup\klam{0}
  .
\end{align}
Suppose that $\bar l(\cdot)=L$ holds $\mu$-almost surely. As shown in
Proposition \ref{theoBarNuRho}, this guarantees that $\rho=\bar\nu$ for the AMS
measure $\nu:=P((Y_i)_{i\in\mathbb{Z}}\in\cdot)$ of the expanded
process (\ref{DefY}).  Thus we have 
\begin{align}
  H_{\bar\nu}(n)=H_{\bar P}(\bar Y_k^{k+n-1})
  \text{ and }
  H_{\bar\mu}(n)=H_P(X_k^{k+n-1})
  ,
\end{align}
where $H_P(U):=\sred_P \kwad{-\log P(U=\cdot)}$ is the entropy of
a~discrete variable $U$.

Denote the conditional entropy $H_P(U|V):=H_P(U,V)-H_P(V)$ and covariance
$\cov_P(U,V):=\sred_P (UV)-\sred_P U\sred_P V$. Entropies of blocks
drawn from the above introduced processes can be linked easily when
blocks of random length are allowed.
\begin{proposition}
  \label{theoHBarNu}
  Suppose that a~process $(X_i)_{i\in\mathbb{Z}}$ is stationary and
  $L=\sred_P\abs{f(X_i)}<\infty$ for a~prefix-free
  $f:\mathbb{X}\rightarrow \mathbb{Y}^*$. Consider then processes
  $(Y_i)_{i\in\mathbb{Z}}$, $(\bar X_i)_{i\in\mathbb{Z}}$, and $(\bar
  Y_i)_{i\in\mathbb{Z}}$ that satisfy (\ref{DefY}), (\ref{DefBarY}),
  (\ref{DefBarX}), and (\ref{DefN}). Put also $M_n:=\sum_{i=1}^n
  \abs{f(X_i)}$, $\bar M_n:=\sum_{i=1}^n \abs{f(\bar X_i)}$, and
  $\eta:= \sred_P \kwad{\frac{\abs{f(X_1)}}{L} \log
    \frac{\abs{f(X_1)}}{L}}\ge 0$.  Then we have
  \begin{enumerate}
  \item 
    $H_P(Y^{M_n})=H_P(X^{n})$,
  \item $H_{\bar P}(\bar X_{k}^{l}) \ge H_P(X_{k}^{l}) - \eta$ if and
    only if $\cov_P(\abs{f(X_1)},-\log P(X_{k}^{l}=\cdot))\ge 0$ for
    $k\le 1\le l$,
  \item 
    $H_{\bar P}(N|\bar X_{k}^{l}) = \log L + \eta$, $k\le 1\le l$,  whereas
  \item 
    $H_{\bar P}(\bar Y^{\bar M_n-\bar M_1},N) \le H_{\bar P}(\bar X^{n}, N)$  
    and  $H_{\bar P}(\bar X_{2}^{n},N) \le H_{\bar P}(\bar Y^{\bar M_{n}}, N)$.
  \end{enumerate}
\end{proposition}
\begin{proof} In (i) and (iv), we use that $f^*$ is an
  injection for a~prefix-free $f$.
  \begin{enumerate}
  \item The claim is true since $Y^{M_n}=f^*(X^{n})$ and
    $X^{n}=\okra{f^*}^{-1}(Y^{M_n})$.
  \item Whenever $\cov_P(\abs{f(X_1)},-\log P(X_{k}^{l}=\cdot))\ge 0$,
    we observe
    \begin{align*}
      H_{\bar P}(\bar X_{k}^{l})
      +
      \eta
      &=
      \textstyle
      \sred_P \kwad{-\frac{\abs{f(X_1)}}{L}\log P(X_{k}^{l}=\cdot)}
      \\
      &\ge 
      \textstyle
      \sred_P \kwad{\frac{\abs{f(X_1)}}{L}}
      \sred_P \kwad{-\log P(X_{k}^{l}=\cdot)}
      =H_P(X_{k}^{l})
      .
    \end{align*}
  \item By (\ref{DefN}).
  \item By (\ref{DefBarY}), the string $\bar Y^{\bar M_n-\bar M_1}$ is
    a~function of $\bar X^{n}$ and $N$, whereas string $\bar
    X_{2}^{n}$ is a~function of $\bar Y^{\bar M_{n}}$ and $N$. Hence
    the claimed inequalities follow.
  \end{enumerate}
\end{proof}

\begin{corollary}
  \label{theoHBarNuFixed}
  Let $\mu$ be AMS on
  $(\mathbb{X}^{\mathbb{Z}},\mathcal{X}^{\mathbb{Z}})$ and
  $f:\mathbb{X}\rightarrow\mathbb{Y}^L$ be a~prefix-free fixed-length
  coding. Then we have
  \begin{align}
    \abs{H_{\bar\nu}(nL)-H_{\bar\mu}(n)}
    \le H_{\bar\mu}(2) + \log L
  \end{align}
  for the measure $\nu=\mu\circ \okra{f^{\mathbb{Z}}}^{-1}$.  
\end{corollary}
\begin{proof}
  In view of Corollary \ref{theoBarBarII}, we may assume without loss
  of generality that $\mu$ is stationary. Let us repeat the
  construction of processes that precedes Proposition
  \ref{theoHBarNu}. Observe that the processes $\bar X_{k}^{l}$ and
  $X_{k}^{l}$ share the same distribution and $\bar M_n=nL$. Thus
  Proposition \ref{theoHBarNu}(iv) yields $H_{\bar P}(\bar Y^{(n-1)L})
  \le H_{\bar P}(\bar X^{n}, N) \le H_P(X^{n-1})+H_{\bar P}(\bar
  X^2,N)$ and $H_P(X^{n-1})=H_{\bar P}(\bar X_{2}^{n}) \le H_{\bar
    P}(\bar Y^{Ln}, N)\le H_{\bar P}(\bar Y^{L(n-1)})+ H_{\bar P}(\bar
  Y^{L},N)\le H_{\bar P}(\bar Y^{L(n-1)})+ H_{\bar P}(\bar X^2,N)$. To
  complete the proof, notice that $H_{\bar P}(\bar
  X^2,N)=H_P(X^2)+\log L$.
\end{proof}

\section{Encoding of the process $X_i=(K_i,Z_{K_i})$}
\label{secLinguistic}

We have not managed to produce an analogue of Corollary
\ref{theoHBarNuFixed} for the processes discussed in Proposition
\ref{theoConj}. But, as shown in \cite{Debowski08c}, then we have 
\begin{align}
  H_{P}(X^n)&
  \ge h_{\mu} n+\kwad{\log 2 -\eta(\delta)} \cdot \abs{U_\delta(n)}
  ,
  \\
  H_{\bar P}(\bar Y^m)&\ge 
  h_{\bar\nu} m+
  \kwad{\log 2 -\eta({\bar\delta})} \cdot \abs{\bar U_{\bar\delta}(m)}
  ,
\end{align}
where $\nu:=\mu\circ \okra{f^{\mathbb{Z}}}^{-1}$ and $\eta(p) :=
-p\log p-(1-p)\log (1-p)$. Whereas $h_{\bar\nu}=L^{-1}h_{\mu}$ by
Proposition \ref{theohBarMuhBarNu}, point (d) of the proof below
demonstrates that $\bar U_{\bar\delta}(m) \supset U_{\delta}(n)$ for
$\delta>{\bar\delta}/a$,
$n=\floor{(\delta-{\bar\delta}/a)L^{-1}(m-C_a)}$,
$a\in({\bar\delta},1)$, and a~certain constant $C_a$.


\begin{proof*}{the Proposition \ref{theoConj}}
  \begin{itemize}
  \item[(a)-(b)] Process $(\abs{f(X_i)})_{i\in\mathbb{Z}}$ is
    ergodic. Thus the expansion rate equals its expectation almost
    surely:
    \begin{align*}
      \bar l((X_i)_{i\in\mathbb{Z}})
      =
      L= \sred_P \abs{f(X_i)} 
      =
      \sum_{k=1}^\infty (\floor{\log_2 k}+2)
      \frac{k^{-\alpha}}{\zeta(\alpha)}  
      \in (0,\infty)
      .
    \end{align*}
    Hence the stationary mean $\overline{\mu\circ
      \okra{f^{\mathbb{Z}}}^{-1}}$ exists by Proposition
    \ref{theoBarNu} and constitutes a~measure over a~finite alphabet.
    
  \item[(c)] Consider the process
    $(Y_i)_{i\in\mathbb{Z}}=f^{\mathbb{Z}}((X_i)_{i\in\mathbb{Z}})$
    and the probability kernel $\tau(\cdot,\omega)=
    P((Y_i)_{i\in\mathbb{Z}}\in\cdot||(Z_k)_{k\in\mathbb{N}})(\omega)$.
    For $\zeta(\alpha)>4$ and $P$-almost all $\omega$,
    $\tau(\cdot,\omega)$ takes form of the measure $\nu$ considered in
    Corollary \ref{theogFEMuFENuInf}.  Hence the distribution of the
    process $(Y_i)_{i\in\mathbb{Z}}$ has finite energy by Proposition
    \ref{theoFENuInt} and, consequently, $(\bar Y_i)_{i\in\mathbb{Z}}$
    has finite energy by Proposition \ref{theoFENuBar}.

  \item[(d)] Define $\bar Z_k:=\bar s_k((\bar Y_i)_{i\in\mathbb{Z}})$.
    The functions $\bar s_k$ are shift-invariant.  Notice that
    $Z_k=\bar s_k((Y_i)_{i\in\mathbb{Z}})$ almost surely on the space
    $(\Omega,\mathcal{J},P)$. Hence, by Eq. (\ref{EqInvar}) applied to
    the AMS measure $\mu\circ \okra{f^{\mathbb{Z}}}^{-1}$, the process
    $(\bar Z_k)_{k\in\mathbb{N}}$ also consists of independent
    equidistributed binary variables measurable against the
    shift-invariant $\sigma$-field of $(\bar Y_i)_{i\in\mathbb{Z}}$.

    Repeat the construction of processes that precedes Proposition
    \ref{theoHBarNu}, putting $M_n:=\sum_{i=1}^n \abs{f(X_i)}$ and
    $\bar M_n:=\sum_{i=1}^n \abs{f(\bar X_i)}$. Recalling that
    $\sred_P M_1=L$, fix such a~$C_a>0$ that
    \begin{align*}
      \sred_P \kwad{M_1 \boole{M_1\le C_a}}\ge a L 
    \end{align*}
    for some $a\in({\bar\delta},1)$.  Observe that  $\bar s_k(\bar
    Y^m)=z$ if $s_k(\bar
    X_2^n)=z\in\klam{0,1}$ and $\bar M_n\le m$. Hence,
    \begin{align*}
      \bar P\okra{\bar s_{k}\okra{\bar Y^m}=\bar Z_k} 
      \ge \bar P\okra{s_{k}\okra{\bar X_2^n}=\bar Z_k, 
        \bar M_1\le C_a, \bar M_n-\bar M_1\le m-C_a}
      .
    \end{align*}
    The event on the right-hand side is measurable $(\bar
    X_i)_{i\in\mathbb{Z}}$ since $\bar Z_k=s_k((\bar
    X_i)_{i\in\mathbb{Z}})$. On the other hand,
    $s_k((X_i)_{i\in\mathbb{Z}})=Z_k$. Thus by (\ref{DefBarX}) and
    further by the independence of the variable $M_1$ from
    $(X_2^n,Z_k)$ we obtain
    \begin{align*}
      &\bar P\okra{s_{k}\okra{\bar X_2^n}=\bar Z_k, 
        \bar M_1\le C_a, \bar M_n-\bar M_1\le m-C_a}
      \\
      &\qquad =
      L^{-1}\sred_P \kwad{ M_1
        \boole{s_{k}\okra{X_2^n}=Z_k, M_1\le C_a, M_n-M_1\le m-C_a} }
      \\
      &\qquad =
      L^{-1}\sred_P \kwad{ M_1 \boole{M_1\le C_a} }
      \sred_P \kwad{ \boole{s_{k}\okra{X_2^n}=Z_k, M_n-M_1\le m-C_a} }
      .
    \end{align*}
    But $(X_i)_{i\in\mathbb{Z}}$ is stationary, so the last expression
    yields simply
    \begin{align*}
      \bar P\okra{\bar s_{k}\okra{\bar Y^m}=\bar Z_k} \ge 
      a
      P\okra{s_{k}\okra{X^{n}}=Z_k, M_{n}\le m-C_a} 
      .
    \end{align*}

    Now, by $P(A\cap B)\ge P(A)-P(B^c)$ and by the Markov
    inequality,
    \begin{align*}
      &P\okra{s_{k}\okra{X^{n}}=Z_k, M_{n}\le m-C_a} 
      \\
      &\qquad \ge
      P\okra{s_{k}\okra{X^{n}}=Z_k}-P\okra{M_{n}> m-C_a}
      \\
      &\qquad \ge
      P\okra{s_{k}\okra{X^{n}}=Z_k}-\frac{Ln}{m-C_a}
      .
    \end{align*}
    Taking $\delta>{\bar\delta}/a$ and
    $n=\floor{(\delta-{\bar\delta}/a)L^{-1}(m-C_a)}$, we obtain
    \begin{align*}
      k\in \bar U_{\bar\delta}(m) 
      &\impliedby 
       a
       \okra{P\okra{s_{k}\okra{X^{n}}=Z_k}
         -(\delta-{\bar\delta}/a)}\ge {\bar\delta}
      \\
      &\iff
      P\okra{s_{k}\okra{X^{n}}=Z_k}\ge \delta
      \iff k\in U_{\delta}(n), 
    \end{align*}
    so (\ref{UDPPowerLawY}) follows for $\beta=\alpha^{-1}$
    from (\ref{UDPPowerLawX}).
  \end{itemize}
\end{proof*}

\section*{Acknowledgements}

The author thanks Peter Harremo\"es, Peter Gr\"unwald, Jan Mielniczuk,
and an anonymous referee for remarks that helped to improve the
quality of this paper. The manuscript was completed during the
author's leave from the Institute of Computer Science, Polish Academy
of Sciences, whereas the research was supported by the grant no.\
1/P03A/045/28 of the Polish Ministry of Scientific Research and
Information Technology and
under the PASCAL II Network of Excellence, IST-2002-506778.

\bibliographystyle{abbrvnat}

\bibliography{0-journals-abbrv,0-publishers-abbrv,ai,ql,mine,tcs,books,nlp}

\end{document}